\documentclass[10pt,reqno]{amsart}
  \usepackage{geometry}
\usepackage{amsmath,mathtools}
\usepackage{amssymb}
\usepackage{caption}
\usepackage{subcaption}

\usepackage{tikz}
\usetikzlibrary{cd}

\newcommand{\OO}{\mathbf{O}}

\newcommand{\N}{\mathbb{N}}
\newcommand{\F}{\mathbb{F}}
\newcommand{\conv}{\diamond}
\newcommand{\lact}{\triangleright}
\newcommand{\ract}{\triangleleft}
\newcommand{\C}{\mathbf{C}}
\newcommand{\Com}{\operatorname{Com}}
\newcommand{\ComC}{\Com (\C)}
\newcommand{\Comag}{\operatorname{Comag}}
\newcommand{\ComagC}{\Comag (\C)}

\newcommand{\cCom}{\operatorname{cCom}}
\newcommand{\cComC}{\cCom(\C)}
\newcommand{\Set}{\mathbf{Set}}
\newcommand{\Cat}{\mathbf{Cat}}

\newcommand{\D}{\mathbf{D}}
\newcommand{\M}{\mathbf{M}}

\newcommand{\Ab}{{\mathbf{Ab}}}
\newcommand{\Aff}{{\mathbf{AffSch}}}
\newcommand{\NCS}{{\mathbf{Ring}^\op}}
\newcommand{\Abop}{\mathbf{Ab}^\mathrm{op}}

\newcommand{\ob}[1]{\mathrm{Ob}\left(#1 \right)}
\newcommand{\fset}[1]{\underline{#1}}
\DeclareMathOperator{\id}{id}
\newcommand{\op}{\text{op}}
\newcommand{\one}{\mathbf{1}}

\theoremstyle{plain}
\newtheorem{theorem}{Theorem}[section]
\newtheorem*{question}{Question}
\newtheorem{proposition}[theorem]{Proposition}
\newtheorem{lemma}[theorem]{Lemma}

\newtheorem{corollary}[theorem]{Corollary}

\theoremstyle{definition}

\newtheorem{definition}[theorem]{Definition}
\newtheorem{remark}[theorem]{Remark}

\begin{document}
\title{Clones from comonoids}
\author{Ulrich Krähmer}
\address{
    TU Dresden,
    Institute of Geometry,
    01062 Dresden,
    Germany
}
\email{ulrich.kraehmer@tu-dresden.de}
\author{Myriam Mahaman}
\address{
    TU Dresden,
    Institute of Geometry,
    01062 Dresden,
    Germany
}
\email{myriam.mahaman@tu-dresden.de}
\begin{abstract}
The fact that the cocommutative comonoids
in a symmetric monoidal category form the
best possible approximation by a 
cartesian category is revisited when 
the original category is only braided
monoidal. This leads to the question when 
the endomorphism operad of a
comonoid is a clone (a Lawvere theory). By
giving an explicit example, we prove that
this does not imply that the comonoid is
cocommutative. 
\end{abstract}
\subjclass[2020]{18M05,18M15,18M65,18M60}
\keywords{clone, cartesian
category, cocommutative comonoid}
\maketitle

\section{Introduction}
Clones are a special type of operads.
Most authors define and study
the concept in the category $\Set$ of sets in
terms of structure
maps and axioms as in
Definition~\ref{defclone}
below, see e.g.~\cite{kerkhoffShortIntroductionClones2014}. However,
if one defines operads
as multicategories with a
single object, then clones are
simply operads which, as a
multicategory, are cartesian. 
In particular, they are equivalent to 
Lawvere theories, see
e.g.~\cite{gouldCoherenceCategorifiedOperadic2008,hylandNotionLambdaMonoid2014,akhvledianiRelatingTypesCategorical2012}. 

The fact that clones are
firmly rooted in the world of
cartesian categories explains why 
they occur naturally in areas
such as logic, set
theory, discrete mathematics
and theoretical computer science.
As cartesian categories can
also be characterised in terms of
comonoid structures on their objects (see
Theorem~\ref{recognition} below), 
we were wondering whether the
endomorphism operads of
suitable comonoids in braided
monoidal categories are
also clones.  
At first, the following 
seems to indicate that one must restrict 
to cocommutative comonoids in symmetric
monoidal categories:

\begin{theorem}\label{notsogood}
Let $\C$ be a braided monoidal 
category, $\Com(\C)$ be the category of
comonoids in $\C$, and
$(X, \Delta _X, \varepsilon _X)$ be a
comonoid. 
Then the monoidal subcategory of 
$\ComC$ formed by the tensor 
powers $X^{\otimes n}$ is
cartesian if and only if $X$ is
cocommutative and the
braiding on $X$ is a symmetry. 
\end{theorem}

In this case, the operation
\begin{equation}\label{clonestr1}
	\varphi \bullet (\psi _1,\ldots,
	\psi _m) \coloneqq 
	\varphi \circ (\psi _1 \otimes 
	\cdots \otimes \psi _m) \circ 
	\Delta^{m-1} _{X^{\otimes n}}
\end{equation}
and the morphisms
\begin{equation}\label{clonestr2}
	\pi _{i,n} \coloneqq
	\varepsilon _X \otimes 
	\cdots \otimes 
	\varepsilon _X \otimes 
	\mathrm{id} _X \otimes	
\varepsilon_X \otimes \cdots 
	\otimes \varepsilon _X 
\end{equation}
define a clone structure on the 
endomorphism operad of $X$ in   
$\ComC$.
Here, $ \varphi \colon X^{\otimes m} \to
X,\psi _j \colon X^{\otimes n} \to
X$ are comonoid morphisms and $ \Delta
_{X^{\otimes n}}$ is the canonical
comultiplication on $X^{\otimes n}$ which
gets applied $m-1$ times; see
Theorem~\ref{important2} for details.

However, the 
endomorphism
operad of $X$ can be cartesian
even if the monoidal category
generated by $X$ is not. 
We will illustrate this by
giving a simple example of a comonoid
in $\Abop$ (i.e.~of a unital associative ring) 
whose endomorphism operad becomes a clone
in the above way although it is not
cocommutative, see
Theorem~\ref{thm:EndAisaclone}. 
So we think it
is interesting to
ask:

\begin{question}
For which comonoids is the
endomorphism operad a clone?
\end{question}

The purpose of
this paper is to raise rather
than to answer this question.  
Approaching it from various
angles has not led to any necessary or
sufficient conditions for such
comonoids. Even for very small
ones, computing their
endomorphism operad explicitly 
is rather intricate, as the
example we give shows, and we
are not aware of a
general method to construct
such examples. 

Although
Theorem~\ref{notsogood} follows easily
from some standard results on cartesian
categories, we also felt it worthwhile to
write a self-contained exposition of this
fact, as examples of
endomorphism clones of cocommutative
comonoids (or of commutative monoids
viewed in the opposite category) 
and their subclones have not been explored
much.
The main ingredient in the proof is the
statement that a cartesian category is the same as a symmetric
monoidal category in which every object is
in a unique way a cocommutative comonoid. 
To the best of our knowledge,
this is originally due to Fox
\cite{foxCoalgebrasCartesianCategories1976} (see
e.g.~\cite[Theorem~4.28]{heunenCategoriesQuantumTheory2019}
for a textbook that discusses
the result). However, all the
references we are aware of start with a
symmetric monoidal category.
As we were interested in a
generalisation to braided monoidal
categories, we reformulate this result in
a way that focuses entirely on comonoids
and comonoid morphisms. The symmetry
of the braiding and the
cocommutativity of the comonoids are
viewed rather as a
side-effect:

\begin{theorem}\label{recognition}
    A monoidal category \(\D\) is
cartesian if and only if there exists a braided monoidal
    category  \(\C\) such that
    \begin{enumerate}
        \item \(\D\) is a monoidal
subcategory of the category \(\ComC\) 
and
        \item
            the counit and comultiplication
            of every object of \(\D\)
            are morphisms in  \(\D\).
    \end{enumerate}
    In this case, the canonical symmetry of \(\D\)
    is the restriction of the
braiding of \(\C\), and all comonoids in
$\D$ are cocommutative.
\end{theorem}

The original version of
this result in \cite{foxCoalgebrasCartesianCategories1976} was stated as an
adjunction: passing from a symmetric
monoidal category \(\C\) to its category
\(\cComC\) of cocommutative comonoids 
defines a right adjoint to the forgetful
functor from the category of cartesian
categories to the category of symmetric
monoidal categories. 
The above theorem reflects the fact that
the forgetful
functor from cartesian to braided monoidal
categories does not have a right adjoint. 
This can be seen for example by noticing
that the direct product of cartesian
categories is also a coproduct in the
category of cartesian categories, but it
is not a coproduct in the category of
braided monoidal categories.
 
Note that there is also a left adjoint to the
forgetful functor from cartesian to
symmetric monoidal categories, which can be constructed
out of the free cartesian category and the free
symmetric monoidal category functors 
on the category $\Cat$ of categories. As
shown by Curien \cite{curienOperadsClonesDistributive2011}, the
corresponding monads on $\Cat$ extend to
the category of profunctors which is
self-dual, and if one considers the monads
as comonads, one obtains a neat uniform characterisation
of clones and symmetric operads as
monoids in their co-Kleisli categories. Using the
free monoidal category functor also yields
an analogous description of non-symmetric
operads; see also \cite{troninAbstractClonesOperads2002} for a
different unified approach to clones and
operads.

Here is a brief outline of the 
paper: the aim of 
Section~\ref{comonoidssec} is to discuss
the relation
between cartesian structures on categories
and comonoid structures on their objects,
starting from semi-cartesian categories (monoidal
categories whose unit object is terminal)
and ending with the proofs of 
Theorems~\ref{notsogood} and
\ref{recognition}.
 The following
Section~\ref{multicatsec} 
recalls the definition of a clone as a
cartesian operad, the action of the
category of finite cardinals
on clones, and the clone
structure on the endomorphism
operad of an object in a
cartesian category. 
Up to here, the paper is rather expository
and does not contain novel results. Its
main new contribution is the example
discussed in detail in the final 
Section~\ref{examplessec} of the paper. 
Here we will consider unital associative,
but not necessarily commutative rings as
comonoids in $ \Abop$. In particular, we give the
example of a noncocommutative
comonoid whose endomorphism
operad is a clone in
Theorem~\ref{thm:EndAisaclone}.
We also prove at the end of 
Section~3 that this phenomenon
can never occur for comonoids
that are Hopf monoids, see
Proposition~\ref{hopf}.  

Throughout the paper we assume the reader is
familiar with basic category
theory including the definitions of
monoidal, braided monoidal and
symmetric monoidal categories
as given e.g.\ in
~\cite{heunenCategoriesQuantumTheory2019}. To shorten the
presentation, we assume
that all monoidal
categories are strict.

\subsection*{Acknowledgements} 
We thank the referees of this
paper for their suggestions,
corrections, and questions. 
Ulrich Krähmer is supported
by the DFG grant
``Cocommutative comonoids''
(KR 5036/2-1).
 
\section{Cocommutative comonoids and Cartesian categories}  
\label{comonoidssec}
Throughout this section,
\((\C,\otimes,\one)\) is a
(strict) monoidal
category. 
The main goal is to 
recall the definition 
of cartesian categories and 
their characterisation in
terms of cocommutative
comonoid structures on their
objects. 

\subsection{Semi-cartesian and cartesian categories}
We begin by considering 
semi-cartesian
categories.
\begin{definition}
One calls \(\C\) 
\emph{semi-cartesian}
if \(\one\) is terminal.
\end{definition}

\begin{proposition}\label{udine}
    A monoidal category is semi-cartesian
    if and only if it admits 
a natural
    transformation \(\varepsilon_X \colon X \to \one\) 
    such that \(\varepsilon_{\one} = \id_{\one}\).
In particular, \(\C\) 
admits at most one such
natural transformation. 
\end{proposition}

\begin{proof}
    See \cite[Proposition 4.15]{heunenCategoriesQuantumTheory2019}.
\end{proof}

\begin{definition}
Let \(\C\) be
semi-cartesian.
\begin{enumerate}
\item The natural transformation 
 \(\varepsilon\) from
Proposition~\ref{udine} is
called the \emph{uniform
deletion} in \(\C\).
\item For any pair of objects 
\(X\) and \(Y\), we denote by
\[ 
	\pi ^1_{XY} \colon 
	X \otimes Y \rightarrow 
	X,\quad
	\pi ^2_{XY} \colon 
	X \otimes Y \rightarrow Y
\]
the \emph{canonical projections}
given by
\[
\begin{tikzcd}
    X &
    X \otimes Y \arrow[l, "\id_{X} \otimes \varepsilon_{Y}"']
    \arrow[r, "\varepsilon_{X} \otimes \id_{Y}"] &
    Y.
\end{tikzcd}
\] 
\end{enumerate}
\end{definition}

In general, these canonical projections do
not have the universal property that 
makes $(X \otimes Y,\pi^1_{XY},\pi
^2_{XY})$ a product of $X$ and $Y$ in 
\(\C\). 

\begin{definition}\label{defcartes}
A \emph{cartesian}
category is a semi-cartesian
category \(\C\) in which  
for any objects \(X,Y\) in  \(\C\),
            the triple  
            \((X \otimes Y, 
	\pi ^1_{XY} , \pi^2_{XY})\)
            is a categorical product 
of \(X\) and \(Y\) in  \(\C\),
so that
for any pair of morphisms \(f \colon 
Z \to X\) and  \(g \colon Z \to Y\),
there is a unique morphism  
\[
	f * g \colon Z \to 
	X \otimes Y
\] 
making the following diagram commutative:
\[
\begin{tikzcd}
    & 
    Z 
    \arrow[rd, "g"] 
    \arrow[ld, "f"'] 
    \arrow[d, dashed, "f*g"] &
    \\
    X &
    X \otimes Y 
    \arrow[r, "\pi^{2}_{XY}"'] 
    \arrow[l, "\pi^{1}_{XY}"] & 
    Y
\end{tikzcd}.
\]  
\end{definition}

\begin{remark}
The above definition is usually
referred to as \emph{cartesian monoidal}
in the literature in order to
distinguish it from the other
uses of the term ``cartesian category''.

\end{remark}


We now shift 
the perspective on this
property: rather than fixing
$X,Y$, we fix an object $Z$ and
show that the universal 
property of $X \otimes Y$
hinges, for morphisms with
domain $Z$, on 
the existence of a
counital comagma
structure on $Z$:
\begin{definition}
A \emph{comagma} in a monoidal
category $\C$ is an object $Z$
together with a morphism 
$ \Delta  \colon Z \rightarrow
Z \otimes Z$. A \emph{counital
comagma} is a comagma together
with a morphism $ \varepsilon
\colon Z \rightarrow
\one $ rendering            
the following diagram commutative:
           \[
                \begin{tikzcd}
                & Z 
                \arrow[d,
"\Delta"] 
                \arrow[ld, "\id_Z"'] 
                \arrow[rd, "\id_Z"] &   \\
                Z & 
                Z \otimes Z 
                \arrow[l,
"\id_Z \otimes \varepsilon"] 
                \arrow[r, "\varepsilon \otimes \id_Z"'] & 
                Z
                \end{tikzcd}
           \] 
\end{definition}

The following lemma addresses
the existence part of the 
universal property of a
categorical product; 
the uniqueness and
naturality will be discussed
afterwards.

\begin{lemma}\label{lem:existcounital}
   Let \(\C\) be a
semi-cartesian category 
   with uniform deletion
\(\varepsilon\).
If a morphism \( \Delta  \colon 
	Z \to Z \otimes Z\)
is counital with respect
to $ \varepsilon _Z$, then 
the maps
\begin{align*}
	*_{X,Y} \colon 
	\C(Z,X) \times 
	\C(Z,Y) &\rightarrow 
	\C (Z,X \otimes Y),
	\\
	(f,g) &\mapsto 
	f*g:=(f \otimes g) \circ
\Delta 
\end{align*}
make the following
            diagrams commutative:
\begin{equation}\label{counitalcomagma}
\begin{tikzcd}
    & 
    Z 
    \arrow[rd, "g"] 
    \arrow[ld, "f"'] 
    \arrow[d, "f*g"] &
    \\
    X &
    X \otimes Y 
    \arrow[r, "\pi^{2}_{XY}"'] 
    \arrow[l, "\pi^{1}_{XY}"] & 
    Y
\end{tikzcd}.
\end{equation}          
This establishes a bijection
between counital comagma
structures on $Z$ and natural
transformations 
$$
	\C(Z,-) \times 
	\C(Z,-) \rightarrow 
	\C(Z,- \otimes -) 
$$
with this property. 
\end{lemma}

\begin{proof}
If $ \Delta $ is counital,
then $ *_{X,Y}$ has the
desired property since 
(\ref{counitalcomagma}) 
expands to
$$
\begin{tikzcd}
& Z 
\arrow[d,
"\Delta"] 
\arrow[ld, "\id_Z"'] 
\arrow[rd, "\id_Z"] &   \\
Z  
\arrow[d, "f"'] &
Z \otimes Z 
\arrow[d, "f \otimes g"] 
\arrow[r, "\pi^{2}_{ZZ}"'] 
\arrow[l, "\pi^{1}_{ZZ}"] & 
Z
\arrow[d, "g"] 
\\
X &
X \otimes Y 
\arrow[r, "\pi^{2}_{XY}"'] 
\arrow[l, "\pi^{1}_{XY}"] & 
Y
\end{tikzcd}.
$$
The naturality of $
*_{X,Y}$ follows
immediately from its
definition. Conversely, 
$ \Delta $ is recovered from
$*$ as $ \Delta = \mathrm{id}
_Z * \mathrm{id} _Z$. 
\end{proof}

Just as we did for semi-cartesian categories,
 we are going to characterise cartesian categories
 in terms of some natural transformations.

 \begin{theorem}
     \label{thm:cartmon_nattrans}
    A semi-cartesian category \(\C\) 
    is cartesian 
    if and only if there exists
    a counital natural transformation
    \(\Delta_{X} \colon X \to X \otimes X \)
    such that
    for any two objects
\(X,Y\), we have
            \begin{equation}
                \label{eq:deltaeps_id}
                (\pi^{1}_{XY}
                \otimes 
                \pi^{2}_{XY})
                \circ \Delta_{X \otimes Y}
                = \id_{X} \otimes \id_{Y}.
            \end{equation}
If such a natural
transformation exists, it is
unique. 
\end{theorem}

 \begin{proof}
``\(\Rightarrow \)'':
     Assume \(\C\) is cartesian. 
     For each object \(Z\),
     set 
\[
	\Delta_Z \coloneqq \id_Z *
\id_Z.
\]
     This morphism is counital
by construction. Furthermore, 
     Lemma~\ref{lem:existcounital}
implies that
     \(f*g = (f \otimes g) \circ \Delta_Z\)
     holds for any \(f \colon Z \to X\) and \(g \colon Z \to Y\).

     Let us show that the family of morphisms
     \((\Delta_Z)_{Z \in \ob{\C}}\) 
     is a natural transformation.
     Given a morphism \(f \colon X \to Y\),
     we deduce from the following
     two commutative diagrams
      \[
     \begin{tikzcd}
         &
         X 
         \arrow[ldd, bend right = 20, "f"'] 
         \arrow[rdd, bend left = 20, "f"] 
         \arrow[d,"f"] &
         \\
         &
         Y 
         \arrow[ld, "\id_{Y}"']
         \arrow[rd, "\id_{Y}"]
         \arrow[d, "\Delta_{Y}"] &
         \\
         Y &
         Y \otimes Y
         \arrow[l, "\pi^{1}_{YY}"]
         \arrow[r, "\pi^{2}_{YY}"'] &
         Y
     \end{tikzcd}
     \qquad \text{and}\qquad
     \begin{tikzcd}[row sep = large]
         &
         X
         \arrow[ldd, bend right = 20, "f"'] 
         \arrow[rdd, bend left = 20, "f"] 
         \arrow[dd, "(f \otimes f) \circ \Delta_{X}"] &
         \\
         &&\\
         Y &
         Y \otimes Y
         \arrow[l, "\pi^{1}_{YY}"]
         \arrow[r, "\pi^{2}_{YY}"'] &
         Y
     \end{tikzcd}
     \] 
     that
    \(
        \Delta_{Y} \circ f 
        = f * f 
        = (f \otimes f) \circ \Delta_{X}
        \),
        hence \(\Delta\) is indeed 
        a natural transformation.

Finally, (\ref{eq:deltaeps_id}) 
holds, as the diagram
     \[
         \begin{tikzcd}[row sep = huge]
         &
         X \otimes Y
         \arrow[ld, bend right = 20, "\pi^{1}_{XY}"']
         \arrow[rd, bend left = 20, "\pi^{2}_{XY}"]
         \arrow[d, "\id_X \otimes \id_Y"] &
         \\
         X &
         X \otimes Y
         \arrow[l, "\pi^{1}_{XY}"]
         \arrow[r, "\pi^{2}_{XY}"'] &
         Y
     \end{tikzcd},
     \] 
commutes
    for any  two objects
\(X,Y\) in  \(\C\),
which means that 
     \[
     \id_X \otimes \id_Y
     = \pi^{1}_{XY}
     *
     \pi^{2}_{XY}
     = (\pi^{1}_{XY}
     \otimes
     \pi^{2}_{XY})
     \circ \Delta_{X \otimes
Y}.
     \]

``\(\Leftarrow\)'': Let us now prove the
converse, so assume there exists a counital 
     natural transformation
     \(\Delta\) satisfying
(\ref{eq:deltaeps_id}).
     By Lemma~\ref{lem:existcounital},
     the morphism \((f \otimes g) \circ \Delta_Z\)
     satisfies the universal property
     for any objects \(X,Y,Z\) and morphisms 
     \(f \colon Z \to X\) and \(g \colon Z \to Y\).
     Let us show 
     that this morphism is unique.
     Let \(h \colon Z \to X \otimes Y\) be a morphism such that
     \(\pi^{1}_{XY} \circ h = f\)
     and \(\pi^{2}_{XY} \circ h = g\).
     Then we have
     \begin{align*}
         h & = 
         (\id_{X} \otimes \id_{Y}) 
         \circ h \\
          & = 
         (\pi^{1}_{XY}
    \otimes
    \pi^{2}_{XY})
     \circ \Delta_{X \otimes Y} 
     \circ h \\
        & = 
        (\pi^{1}_{XY}
     \otimes 
     \pi^{2}_{XY})
     \circ 
     (h \otimes h) \circ \Delta_{Z} \\
           & = 
           [(\pi^{1}_{XY} \circ h)
           \otimes 
           (\pi^{2}_{XY} \circ h)]
     \circ \Delta_{Z} \\
           & = (f \otimes g) \circ
\Delta_{Z}.
     \end{align*} 
In particular, taking $ f = g
= \mathrm{id} _Z$ shows that $
\Delta _Z$ is unique. 
 \end{proof}

Especially in the context of
theoretical computer science,
this is often rephrased 
by saying that a cartesian
category is a semi-cartesian category
with a natural
\emph{uniform copying}
operation $ \Delta $. For
example, the no-cloning
theorem (quantum computers cannot 
copy information \cite[Theorem
4.27]{heunenCategoriesQuantumTheory2019})
implies that the monoidal
category of finite-dimensional Hilbert
spaces is
not cartesian.


\subsection{Categories of counital comagmas}
We now study the category of counital comagmas further.  
\begin{definition}
    A \emph{morphism of comagmas}
    \(f \colon (X,\Delta_{X}) \to (Y, \Delta_{Y})\)
     is a morphism \(f \colon
X \to Y\) in \(\C\)
     which makes the following diagram commutative:
     \[
     \begin{tikzcd}
          X 
          \ar[d, "f"']
          \ar[r, "\Delta_{X}"]
          &
          X \otimes X 
          \ar[d, "f \otimes f"]
          \\
          Y 
          \ar[r, "\Delta_{Y}"']
          &
          Y \otimes Y.
     \end{tikzcd}
     \]
    If
 \(X\)  and \(Y\) are counital, with counits
     \(\varepsilon_{X}\) and \(\varepsilon_{Y}\) respectively,
     then \(f\) is \emph{counital} if
     the following diagram is commutative:
     \[
         \begin{tikzcd}
         X 
         \ar[r,"f"]
         \ar[dr,"\varepsilon_{X}"']
         &
         Y
         \ar[d,"\varepsilon_{Y}"]
         \\
           &
        \one.
         \end{tikzcd}
     \]
We denote the category of all
counital comagmas by 
\(\ComagC\). 
\end{definition}


Note that in general, the
tensor product of two comagmas
carries no canonical comagma
structure. However, 
\(\ComagC\) becomes monoidal
if $ \C$ is
braided monoidal 
(see
e.g.~\cite[Definition
1.17]{heunenCategoriesQuantumTheory2019}
for more background on braided
monoidal categories):  

\begin{definition}
A \emph{braiding} on $\C$ is 
a natural isomorphism
$$
	\sigma _{X,Y} \colon 
	X \otimes Y \rightarrow Y
	\otimes X
$$ 
which is monoidal
both in $X$ and $Y$, that is,
for which 
\begin{align*}
	\sigma _{X \otimes Y,Z} 
&=
	(\sigma _{X,Z} \otimes
\id_Y) \circ 
	(\id_X \otimes
\sigma_{Y,Z}) \\
	\sigma _{X,Y \otimes Z} 
&=
	(\id_Y \otimes
\sigma_{X,Z}) \circ 
	(\sigma _{X,Y} \otimes
\id_Z)  
\end{align*}
holds. A \emph{symmetric monoidal
category} is a braided
monoidal category whose
braiding is a \emph{symmetry},
meaning that 
for all objects \(X\) and
\(Y\), we have  
 \[
	\sigma_{Y,X} \circ \sigma_{X,Y} = 
	\id_{X \otimes Y}.
\] 
\end{definition}

From now on we assume that 
$\C$ is braided monoidal. 

\begin{proposition}
The braiding $
\sigma $ on $\C$
induces a unique monoidal structure
on $\ComagC$ such that the
forgetful functor 
$\ComagC \to \C$ is strict
monoidal. 
\end{proposition}
\begin{proof}
This is well-known; 
the tensor product $X \otimes Y$
 of two counital comagmas
$(X,\Delta _X, \varepsilon _X)$ and 
$(Y, \Delta _Y , \varepsilon _Y)$ becomes
a counital comagma whose counit 
and comultiplication are 
given as follows: 
\begin{align}
    \varepsilon_{X \otimes Y} & \coloneqq  
    \varepsilon_{X} \otimes \varepsilon_{Y},
    \quad \text{and} \\
    \Delta_{X \otimes Y} &  \coloneqq 
    \left( 
       \id_{X} \otimes \sigma_{X,Y} \otimes \id_{Y}
    \right) 
       \otimes 
       \left( 
         \Delta_{X} \otimes \Delta_{Y}
     \right) \label{tensorcomagma}
.\qedhere
\end{align}
\end{proof}
 
We have moved some proofs to
the appendix, where we use the
standard graphical calculus of
string diagrams. For example,
the above counit and comultiplication
on $X \otimes Y$ are given in 
(\ref{counittensor}). 

\begin{proposition}
    The category \(\ComagC\)
    is semi-cartesian. 
\end{proposition}
\begin{proof}
The unit object $\one$ is easily seen to
carry a unique structure of a counital comagma, and     
the counit 
    \(\varepsilon_{X} \colon X \to \one\)
of any counital comagma
is a morphism of counital comagmas.
    This defines a uniform deletion 
in $\ComagC$. 
\end{proof}

Here are two properties of the
tensor product of comagmas
that will be used later:

\begin{proposition}\label{twopropertiesprop}
    Let \(X,Y\) be two counital comagmas
in $\C$.
    The following equalities 
    hold in \(\C\):
    \begin{enumerate}
        \item \label{delta_identity}
    \(
        [(\id_X \otimes \varepsilon_Y) 
       \otimes 
       (\varepsilon_X \otimes \id_Y)]
       \circ
       \Delta_{X \otimes Y}
        = \id_X \otimes \id_Y.
    \)
       \item
           \(
        [ (\varepsilon_X \otimes \id_Y)
       \otimes 
       (\id_X \otimes \varepsilon_Y)]
       \circ
       \Delta_{X \otimes Y}
       = \sigma_{X,Y}.
           \)

    \end{enumerate}
\end{proposition}
\begin{proof}
See (\ref{twoproperties}) in the appendix.
\end{proof}




So we have seen that for a 
braided monoidal category
$\C$, the category $\ComagC$
is monoidal. However, 
in general it is not braided
monoidal:

\begin{proposition}\label{braidedmonoidalcomag}
    Let \(X,Y\) be two counital comagmas in  \(\C\).
    The braiding  \(\sigma_{X,Y}: X \otimes Y \to Y \otimes X\)
     is a morphism of counital comagmas 
     if and only if \(\sigma_{Y,X} = \sigma_{X,Y}^{-1} \).
\end{proposition}
\begin{proof}
    See p\pageref{prop:notbraided}
    in the appendix.
\end{proof}


\subsection{Cocommutative comonoids}

In order to prepare for the
characterisation of cartesian
categories we collect here
everything we need to know
about cocommutative comonoids.

\begin{definition}
We denote by $\ComC \subseteq
\ComagC$ the full subcategory
of \emph{comonoids}, that is, 
of counital comagmas 
\((X, \Delta, \varepsilon)\)
for which $\Delta$ is
\emph{coassociative},
$$
	( \Delta \otimes
\mathrm{id} _Z ) \circ \Delta
= 
	(\mathrm{id} _Z \otimes
\Delta ) \circ \Delta .
$$
We
denote by 
$\cComC \subseteq \ComC$ the
full subcategory of comonoids
which are
\emph{cocommutative},
\[
  \Delta = \sigma_{X,X} \circ \Delta.
\]
\end{definition} 

We thus 
have the following sequence of forgetful functors 
\begin{equation}
    \label{inclusion}
    \cComC \subseteq \ComC \subseteq \ComagC \to \C,
\end{equation}
where the two inclusions are full.

\begin{proposition}
\(\ComC\) is a semi-cartesian subcategory of \(\ComagC\).
\end{proposition}
\begin{proof}
The unit object is a comonoid,
and \(\ComC\) contains all the terminal morphisms
since it is a full subcategory.
Finally, it is straightforward
to verify that the tensor
product 
of two comonoids in 
$\ComagC$ is again 
coassociative.
\end{proof}

The main ingredient
for the
proof of Theorem~\ref{recognition}
is the next proposition which
is known as the
\emph{Eckmann-Hilton
argument}. In the introduction
we said that we view
cocommutativity
rather as a
side-effect, and this is where 
this statement becomes
manifest:
\begin{proposition}\label{hiltoneckmann}
  Let 
  \(
    X_{1} \coloneqq~(X, \Delta_{1}, \varepsilon_{1})
    \) and 
\(X_{2} \coloneqq~(X, \Delta_{2}, \varepsilon_{2})
  \)
   be two counital comagma structures
   defined
   on the same underlying object \(X\).
   If \(\Delta_{2}\) is a morphism of
   comagmas \(X_{1}
\rightarrow X_{1} \otimes X_{1}\),
   then 
\[
	\varepsilon_{1} =
	\varepsilon_{2}, 
	\quad
   \Delta_{1} =
   \Delta_{2},
\]
      and \(X_1=X_2\)
      is a cocommutative
comonoid.
\end{proposition}
\begin{proof}
    See
p\pageref{eckmannhiltonproof} in the
appendix.
\end{proof}

In other words, we have
\[
	\Comag(\ComagC) = \cComC.
\]

The following is a direct corollary 
of the Eckmann-Hilton argument: 
\begin{corollary}
  \label{comag_comag}
  Let \(X\) be a counital comagma.
  The comultiplication  \(\Delta_{X}\) 
  is a morphism of counital comagmas if and only if
  \(X\) is cocommutative.
  In this case,  \((X, \Delta_{X}, \varepsilon_{X})\) 
  is the unique counital comagma structure on \(X\)
  in  \(\ComagC\).
\end{corollary}
\begin{proof}
The ``only if'' direction follows directly from the Eckmann-Hilton
argument. The ``if'' part is shown 
in Proposition \ref{prop:hilton-eckmann-appendix2}
in the appendix.
\end{proof}


Thus $\cComC$ is the category
of counital comagmas in a
monoidal category which is not
braided. Viewed from this
perspective, the following 
extension of 
Proposition~\ref{braidedmonoidalcomag}
from comagmas to cocommutative
comonoids is rather natural,
albeit surprising at first (as
pointed out by 
Baez
\cite[Lemma~3]{baezHochschildHomologyBraided1994}):

\begin{proposition}
  \label{tensor_cocommutative}
    The tensor product 
    of two cocommutative comonoids
    \(X,Y\)
    is cocommutative if and only if
    \(\sigma_{Y,X} = \sigma_{X,Y}^{-1}\).
\end{proposition}
\begin{proof}
    This is a direct corollary of 
    Proposition~\ref{prop:tens-prop-cocommutative}.
\end{proof}



\subsection{Recognition theorem}

We will
now discuss the canonical
symmetry on a cartesian
category and then 
prove Theorems~\ref{notsogood} 
and~\ref{recognition}. 
 
\begin{proposition}
     \label{cor:cart_braiding}
    If \(\C\) is cartesian,
then the natural transformation
     \[
	\sigma_{X,Y} \coloneqq 
     \pi^{2}_{XY}
     *
     \pi^{1}_{XY}
\]
is the
unique braiding on \(\C\), and
is a symmetry.
\end{proposition}
 \begin{proof}
     The morphisms
     \(\sigma_{X,Y} \) are defined using
     the composition and tensor products
     of natural
transformations.
     They are therefore natural in $X$ and $Y$.
Using the universal property,
it is straightforward to
verify that
     \(\sigma\) is a symmetry. 

      If \(\tau\) is any braiding
     on \(\C\),
     then unit constraints force
     the diagram
\[
  \begin{tikzcd}
  & X \otimes Y 
  \ar[dl, "\pi^{2}_{XY}"']
  \ar[d,"\tau_{X,Y}"]
  \ar[dr, "\pi^{1}_{XY}"]
  & \\
    Y &
    Y \otimes X 
  \ar[l, "\pi^{1}_{YX}"]
  \ar[r, "\pi^{2}_{YX}"']
      &
    X
  \end{tikzcd}
\]
to commute, hence 
\(\tau_{X,Y} = \pi^{2}_{XY} * \pi^{1}_{XY} = \sigma_{X,Y}\)
by the universal property.
 \end{proof}

\begin{proposition}
\label{deltamonoidal}
  Let \(\C\) be cartesian,
  \(\Delta\) be the uniform
copying, and  
  \(\sigma\) be the canonical symmetry.
  Then, for any  \(X,Y \in \ob{\C}\) we have
  \[
    \Delta_{X \otimes Y} = 
    \left( 
       \id_{X} \otimes \sigma_{X,Y} \otimes \id_{Y}
    \right) 
       \otimes 
       \left( 
         \Delta_{X} \otimes \Delta_{Y}
       \right).
  \]
  
\end{proposition}
\begin{proof}
Consider the following diagram:
\[
  \begin{tikzcd}[row sep = large, column sep = large]
  & X \otimes Y 
  \ar[ld, "\id_{X \otimes Y}"'] 
  \ar[rd, "\id_{X \otimes Y}"] 
  \ar[d, "\Delta_{X} \otimes \Delta_{Y}"]
  & \\
    X \otimes Y 
  \ar[d, "\id_{X \otimes Y}"']
  & (X \otimes X) \otimes (Y \otimes Y) 
  \ar[l, "\pi_{XX}^{1} \otimes \pi_{YY}^{1}"] 
  \ar[r, "\pi_{XX}^{2} \otimes \pi_{YY}^{2}"'] 
  \ar[d, "\id_{X} \otimes \sigma_{X,Y} \otimes \id_{Y}"]
  & X \otimes Y 
  \ar[d, "\id_{X \otimes Y}"] \\
    X \otimes Y &
    (X \otimes Y) \otimes (X \otimes Y) 
    \ar[r, "\pi_{X \otimes Y, X \otimes Y}^{1}"'] 
    \ar[l, "\pi_{X \otimes Y, X \otimes Y}^{2}"] 
    & X \otimes Y
  \end{tikzcd}.
\]
The first row in this diagram commutes because of the universal property of \(\Delta_{X}\) and \(\Delta_{Y}\),
whereas the second row commutes because of the universal property of \(\sigma_{X,Y}\).
The morphisms along the border of this diagram 
actually form the diagram defining the universal property of \(\Delta_{X \otimes Y}\):
\[
 \begin{tikzcd}[row sep = large, column sep = large]
  & 
   X \otimes Y 
   \ar[ld, "\id_{X \otimes Y}"']
   \ar[rd, "\id_{X \otimes Y}"]
   \ar[d, dashed, "\Delta_{X \otimes Y}"]
  & \\
   X \otimes Y & 
   (X \otimes Y) \otimes (X \otimes Y) 
   \ar[l, "\pi^{1}_{X \otimes Y, X \otimes Y}"]
   \ar[r, "\pi^{2}_{X \otimes Y, X \otimes Y}"']
  & 
   X \otimes Y
 \end{tikzcd}, 
\]
thus the equality
\[
    \Delta_{X \otimes Y} = 
    \left( 
       \id_{X} \otimes \sigma_{X,Y} \otimes \id_{Y}
    \right) 
       \otimes 
       \left( 
         \Delta_{X} \otimes \Delta_{Y}
       \right).
	\qedhere
\]
\end{proof}


\begin{proposition}
\label{cartesian_comagmas}
A braided monoidal category 
is cartesian if and only if
the forgetful functor
 \(\ComagC \to \C\) is an isomorphism. 
In this case, we have
\[
  \cComC = \ComC = \ComagC \cong \C.
\]
\end{proposition}
\begin{proof}
``\(\Rightarrow\)'' 
Suppose \( \C \) is cartesian.
Then the uniform copying \(\Delta\) and uniform deletion \(\varepsilon\)
induce a functor  \(G \colon \C \to \ComagC\)
\[
  G(X) \coloneqq (X, \Delta_{X}, \varepsilon_{X}),
  \quad
  G(f) \coloneqq f
\]
which is a section of the forgetful functor
\(\ComagC \to \C\).

The previous proposition tells us 
that \(G\) is strict monoidal.
Then for any counital comagma 
\((X,\delta,\epsilon)\) in \(\C\),
its image \(\big(G(X), \delta, \varepsilon\big)\)
 is a counital comagma in  \(\ComagC\).
It follows from Corollary~\ref{comag_comag} 
that the structure maps \(\delta\) and  \(\epsilon\) 
must be equal to those of
of \(G(X) = (X, \Delta_{X}, \varepsilon_{X})\).
Hence \(G(X)\) is the unique counital comagma structure on \(X\),
and \(G\) is an isomorphism. 

Another consequence of
Corollary~\ref{comag_comag} 
is that \(G(X)\)
is a cocommutative comonoid for each \(X \in \ob{\C}\),
therefore
\[
  \cComC = \ComC = \ComagC \cong \C.
\]

``\(\Leftarrow\)'' 
Suppose the forgetful functor \(\ComagC \to \C\)
is an isomorphism.
Since \(\ComagC\) is semi-cartesian, then so is \(\C\).

For each \(X\) in \(\C\), 
let us denote by  \((X,\Delta_{X},\varepsilon_{X})\) 
the image of \(X\)
by the inverse of the forgetful functor  \(\ComagC \to \C\).
Then, the maps \(X \mapsto \Delta_{X}\) 
and \(X \mapsto \varepsilon_{X}\)
induce natural transformations
which
satisfy all the conditions 
in Theorem \ref{thm:cartmon_nattrans},
thus \(\C\) is cartesian.
\end{proof}

As a consequence, 
in a cartesian category,
every object 
     is in a unique and
canonical way a
cocommutative comonoid, and
     every morphism in $\C$ 
is a morphism of comonoids.


From now on, we assume that
\(\C\) is braided monoidal
with a fixed braiding  \(\sigma\).
Our goal is to characterise the subcategories of
\(\ComagC\) which are cartesian.

\begin{proposition}
\label{cartesian_subcategory}
A monoidal subcategory \(\D\) of \(\ComagC\) is cartesian
if and only if for any object \((X,\Delta,\varepsilon)\)
in  \(\D\), the morphisms  \(\Delta\) and \(\varepsilon\)
are in  \(\D\).
\end{proposition}
\begin{proof}
``\(\Leftarrow\)''
The maps which associate to each comonoid its
comultiplication and counit induce a uniform copying
and deletion in \(\D\), thus \(\D\) is cartesian.

``\(\Rightarrow\)''
Suppose \(\D\) is cartesian,
and let \((X,\Delta,\varepsilon)\) be an object in  \(\D\).
Since \(\D\) is cartesian, then  \((X,\Delta,\varepsilon)\)
is a counital comagma in  \(\D\).
We know however from
Corollary~\ref{comag_comag},
that this comagma structure is uniquely given by  \(\Delta\) and
\(\varepsilon\),
thus  these two morphisms are in \(\D\).
\end{proof}

It follows that neither \(\ComagC\)
nor  \(\ComC\) are cartesian in general.
Even though  \(\cComC\) satisfies the conditions
    in the proposition, it is not cartesian either,
since it is not closed under the tensor product in general.
We have however the following:

\begin{corollary}
    If \(\C\) is symmetric monoidal,
    then \(\cComC\) is cartesian.
\end{corollary}
\begin{proof}
    If \(\sigma\) is a symmetry,
    then  \(\cComC\) is a monoidal
    subcategory of  \(\ComC\),
    and hence satisfies the
    conditions in 
Proposition~\ref{cartesian_subcategory}.
\end{proof}

This induces the adjunction
    between symmetric monoidal categories 
    and cartesian categories described
originally in~\cite{foxCoalgebrasCartesianCategories1976}.

We are now ready to prove Theorems~\ref{notsogood} and~\ref{recognition}.

\begin{proof}[Proof of
Theorem~\ref{notsogood}]
Let \(X\) be a comonoid
and \(\D\) be the full subcategory of  \(\ComC\)
whose objects are the tensor powers of  \(X\).
By construction, the category \(\D\) is monoidal.

Proposition \ref{cartesian_subcategory} tells us
that \(\D\) is cartesian if and only if \(\Delta_{X^{\otimes n}}\)
and \(\varepsilon_{X^{\otimes n}}\) are in \(\D\)
for all  \(n \ge 0\).
Since \(\D\) is full, it contains all
the terminal morphisms
\(\varepsilon_{X^{\otimes n}}\) 
and the previous condition becomes equivalent
to \(\Delta_{X^{\otimes n}}\) 
being a morphism of comonoids for all \(n \ge 0\).
This means that \(\Delta_{X^{\otimes n}}\) 
has to be cocommutative for
all \(n \ge 0\) (Corollary~\ref{comag_comag}),
and this in turn is equivalent to \(\Delta_{X}\) being cocommutative
and \(\sigma_{X,X} = \sigma_{X,X}^{-1}\) (Proposition
\ref{tensor_cocommutative}).
\end{proof}


\begin{proof}[Proof of
Theorem~\ref{recognition}]
The ``if'' statement 
follows immediately from 
Proposition~\ref{cartesian_subcategory},
and the ``only if'' statement from
Proposition ~\ref{cartesian_comagmas}.
\end{proof}

\section{Cartesian operads and
clones}\label{multicatsec}
We retain the assumption that $\C$ is a
strict monoidal category. 
Our main goal in this section is to
establish an equivalence between cartesian
operads and clones.

\subsection{Operads}
We will define operads as certain
multicategories. We refer 
to~\cite{leinsterHigherOperadsHigher2004}
for further information. 

\begin{definition}
A \emph{(plain) multicategory} $\M$
consists of
\begin{enumerate}
\item a class $\ob{\M}$ of
\emph{objects}, 
\item for any
$A_1,\ldots,A_n,B \in \ob \M$ 
a class $\M(A_1,\ldots,A_n;B)$ of
\emph{morphisms} and 
\item for any
$A_{11},\ldots,A_{nm_n},B_1,\ldots,B_n,C
\in \ob \M$ of
a \emph{composition}
\begin{align*}
	& \M(B_1,\ldots,B_n;C) \times 
	\M(A_{11},A_{12},\ldots,A_{1m_1}; 
	B_1) \times \cdots \\
	& \quad \cdots \times
	\M(A_{n1},A_{n2},\ldots,A_{nm_n}; 
	B_n) \\
	& \quad \quad \longrightarrow 
	\M(A_{11},A_{12},\ldots,A_{nm_n};
	C)\\
& (\varphi , \psi _1,\ldots,\psi_n) 
	\mapsto \varphi \circ (\psi
_1,\ldots,\psi _n),  
\end{align*}
\item for each $A \in \ob\M$ an 
\emph{identity morphism} 
$ \mathrm{id} _A \in \M(A;A)$
\end{enumerate} 
subject to the obvious associativity and
identity axioms. 
\end{definition}
\begin{definition}
    A \emph{(plain) operad} is a
multicategory $\OO$ with a single object
$X$. 
In this case, we denote 
\( \OO( \underbrace{X, \ldots, X}_{n
\text{ times}}; X)\) simply 
by \(\OO_n\).
\end{definition}

Given a collection of objects 
$M_0 \subseteq \ob \C$ (where $\C$ is as
elsewhere a monoidal category), we define 
the multicategory 
\(\M\) associated to \(M_0\) as follows: 
    \begin{enumerate}
        \item its class of objects is
\(\ob \M \coloneqq M_0\)
        \item for \(A_1,\ldots, A_n,B \in M_0\)
             \[
                 \M(A_1,\ldots,A_n; B)
                 \coloneqq \C(A_1 \otimes
\cdots \otimes A_n, B),
            \] 
        \item the composition is given by
            \(
                \varphi \circ (\psi_1,\ldots,\psi_n)
                \coloneqq 
                \varphi \circ (\psi_1
\otimes \cdots \otimes \psi_n),
            \) and
\item the identity morphisms are those
from $\C$.
    \end{enumerate}
    If \(M_0 = \ob{\C}\), then \(\M\) is called
    the \emph{underlying multicategory} of \(\C\).
    If \(M_0 = \{X\}\),
    then \(\M\) is called
    the \emph{endomorphism operad} of \(X\).

Before we define the analogue of
cartesian categories for multicategories,
we will need to introduce some notions
related to maps between finite sets.

\subsection{The category of finite cardinals}
Let \(\F\) be the category of finite cardinals,
whose objects are the finite sets
\(\fset{n} \coloneqq \{1, \ldots, n\}\),
\(n \in \N\),
and whose morphisms are all possible
maps between these sets.

Addition induces a strict monoidal structure on  \(\F\)
that we denote by \(\oplus\), where the
unit object is the empty set \(\fset 0 = \emptyset\).
Given two morphisms \(f_1 \colon \fset m_1 \to \fset n_1\)
and \(f_2 \colon \fset m_2 \to \fset n_2\),
their tensor product is 
\[
    f_1 \oplus f_2: \fset m_1 \oplus \fset m_2
\to \fset n_1 \oplus \fset n_2
\]
where
\[
    (f_1 \oplus f_2)(i) 
    \coloneqq 
    \begin{cases}
        f_1(i) & 1 \le i \le m_1,\\
	f_2(i-m_1) + n_1 & 
        m_1 < i \le m_1+m_2.
    \end{cases}
\]

We are actually interested in
the opposite category \(\F^\op\). We will
call a morphism $f \colon \fset n \to
\fset m$ in $\F^\op$ a \emph{selection} to clearly
distinguish it from the corresponding map 
$\fset m \to \fset n$; 
we picture $f$ as a way to select an
$m$-tuple with entries taken from a given
$n$-tuple, as in
Figure~\ref{figureselections}. 

\begin{figure}[h]
    \captionsetup[subfigure]{labelformat=empty}
    \tikzset{every picture/.style={scale=0.3, every picture/.style={}}}
    \begin{subfigure}[b]{0.3\textwidth}
    \centering
    \begin{tikzpicture}
\tikzset{
    every node/.style={
        inner sep=0pt,
        minimum size=6pt
    },
    Input/.style={
        fill=black,
        shape=circle,
        label={left:#1}
    },
    Output/.style={
        fill=black,
        shape=circle,
        label={right:#1}
    }
}
\node[Input=1] (0) at (-2, 1) {};
\node[Input=2] (1) at (-2, -1) {};
\node[Output=1] (2) at (1, 2) {};
\node[Output=1] (3) at (1, 0) {};
\node[Output=2] (4) at (1, -2) {};
\draw (0) to (2);
\draw (0) to (3);
\draw (1) to (4);
\end{tikzpicture} 
\caption{\(f \colon \fset{2} \to \fset{3}\)}
    \end{subfigure}
    \hfill
    \begin{subfigure}[b]{0.3\textwidth}
    \centering
    \begin{tikzpicture}
\tikzset{
    every node/.style={
        inner sep=0pt,
        minimum size=6pt
    },
    Input/.style={
        fill=black,
        shape=circle,
        label={left:#1}
    },
    Output/.style={
        fill=black,
        shape=circle,
        label={right:#1}
    }
}
\node[Input=1] (0) at (-2, 2) {};
\node[Input=2] (1) at (-2, 0) {};
\node[Input=3] (2) at (-2, -2) {};
\node[Output=2] (3) at (1, 2) {};
\node[Output=1] (4) at (1, 0) {};
\node[Output=2] (5) at (1, -2) {};
\draw  (1) to (3);
\draw  (1) to (5);
\draw  (0) to (4);
\end{tikzpicture}
     
\caption{\(g \colon \fset{3} \to \fset{3}\)}
    \end{subfigure}
    \hfill
    \begin{subfigure}[b]{0.3\textwidth}
    \centering
    \begin{tikzpicture}
\tikzset{
    every node/.style={
        inner sep=0pt,
        minimum size=6pt
    },
    Input/.style={
        fill=black,
        shape=circle,
        label={left:#1}
    },
    Output/.style={
        fill=black,
        shape=circle,
        label={right:#1}
    }
}
		\node[Input=1] (0) at (-2, 1) {};
		\node[Input=2] (1) at (-2, -1) {};
		\node[Output=1] (2) at (1, 2) {};
		\node[Output=1] (3) at (1, 0) {};
		\node[Output=1] (4) at (1, -2) {};
		\draw  (0) to (2);
		\draw  (0) to (4);
		\draw  (0) to (3);
\end{tikzpicture}
     
    \caption{\(g \circ f\)}
    \end{subfigure}
    \caption{Selections}\label{figureselections}
\end{figure}

For each \(n \in \N\), we denote by
\(\varepsilon_{\fset n} \colon
\fset n \to \fset 0 \)
the selection corresponding to the empty
map and by \(\Delta_{\fset n} \colon
\fset n \to \fset n \oplus \fset n\) the
selection which selects the tuple 
$(1,2,\ldots,n,1,2,\ldots,n)$. 
These morphisms define 
a uniform deletion \(\varepsilon\) 
and a uniform copying \(\Delta\) 
on the category \(\F^\op\),
making this category cartesian.

\subsection{The substitution product}

We assume in this subsection that 
\(\C\) is a cartesian category.
For each \(1 \le i \le n\),
the \(i\)-th canonical projection
\(X_1 \otimes \cdots \otimes X_n \to X_i\)
is given by
\[
\pi^{i}_{X_1, \ldots, X_n}
\coloneqq 
\varepsilon_{X_1}
\otimes \cdots \otimes
\varepsilon_{X_{i-1}}
\otimes \id_{X_i} \otimes
\varepsilon_{X_{i+1}}
\otimes \cdots \otimes
\varepsilon_{X_n}.
\] 
Given a selection 
\(f \colon \fset n \to \fset m\),
we use these projections to construct 
a canonical morphism
\[
    \pi^{(f)}_{X_1, \ldots, X_n}
    \colon
    X_1 \otimes \cdots \otimes X_n
    \to
    X_{f[1]} \otimes \cdots \otimes X_{f[m]},
\] 
by setting 
\[
    \pi^{(f)}_{X_1, \ldots, X_n}
    \coloneqq 
    \pi^{f[1]}_{X_1, \ldots, X_n}
    * \cdots *
    \pi^{f[m]}_{X_1, \ldots, X_n},
\] 
where \(*\) is as in
Definition~\ref{defcartes} and $f[i]$
denotes the \(i\)-th entry of the
tuple corresponding to the selection 
\(f \colon \fset n \to \fset m\).

The following is verified by direct
computation: 

\begin{proposition}
    The operations $ \pi $ have the following properties:
    \begin{enumerate}
        \item \(\pi^{(\id_{\fset n})}_{X_1, \ldots, X_n}
                = \id_{X_1 \otimes \cdots \otimes X_n}\),
            \item \(\pi^{(g \circ f)}_{X_1, \ldots, X_n}
                = \pi^{(g)}_{X_{f[1]},
\ldots, X_{f[m]}}
                \circ \pi^{(f)}_{X_1, \ldots, X_n}\)
        \item
        \(\pi^{(f_1 \oplus f_2)}_{X_1, \ldots, X_{n_1},
        Y_1,\ldots,Y_{n_2}}
        = \pi^{(f_1)}_{X_1,\ldots, X_{n_1}}
        \otimes \pi^{(f_2)}_{Y_{1},\ldots,Y_{n_2}}\)
        \item 
            \(\pi^{(\Delta_{\fset n})}
            _{X_1,\ldots, X_n} = 
            \Delta_{X_1 \otimes \cdots \otimes X_n}\)
        \item 
            \(\pi^{(\varepsilon_{\fset n})}
            _{X_1,\ldots,X_n}
            = \varepsilon_{X_1\otimes 
            \cdots \otimes X_n}.\)
        \item 
            \(\pi^{(\pi^{i} _{\fset 1, \ldots, \fset 1})}
            _{X_1, \ldots, X_n}
            = \pi^i_{X_1, \ldots, X_n}\)
    \end{enumerate}
\end{proposition}

This means that the operations $ \pi $
define a form of action of $\F^\op$: 

\begin{definition}
Let \(f \colon \fset n \to \fset m\)
be a selection. Given morphisms 
\(g_i \colon X_i \to Y_i\)
in $\C$, 
\((1 \le i \le n)\), we define
the \emph{substitution} 
of the \(g_i\)'s in \(f\)
to be the morphism
\[
    f \wr (g_1, \ldots, g_n) 
    \colon
    X_1 \otimes \cdots \otimes X_n
    \to
    Y_{f[1]} \otimes \cdots \otimes
Y_{f[m]}
\] 
given by
\[
    f \wr (g_1, \ldots, g_n) 
    \coloneqq 
    (g_{f[1]} \otimes \cdots \otimes g_{f[m]})
    \circ
    \pi^{(f)}_{X_1, \ldots, X_n}.
\] 
\end{definition}

When applying this to the 
cartesian category \(\C = \F^\op\),
the substitution
\(f \wr (g_1, \ldots, g_n) \)
becomes an internal operation
on the morphisms in $\F^\op$.
Under the composition and this operation,
the set of all morphisms in
\(\F^\op\) is generated 
by the deletion \(\varepsilon _{\fset 1}\),
the duplication \(\Delta_{\fset 1}\)
and the identity
\(\mathrm{id} _{\fset 1}\).

\begin{remark}
The category \(\F^\op\),
    together with the substitution product,
    is a particular example of a \emph{cartesian club},
    as introduced by Kelly in \cite{kellyManyvariableFunctorialCalculus1972,kellyAbstractApproachCoherence1972}.
    In \cite{blackwellTwodimensionalMonadTheory1989}, this club was viewed
as a \(2\)-monad whose algebras are exactly the categories
    with finite products.
\end{remark}

    \subsection{Cartesian operads}
    The following material can
be found for example in
\cite[Section~2.3]{gouldCoherenceCategorifiedOperadic2008},
\cite[Section~2.6]{shulmanCategoricalLogicCategorical2016}
or \cite{shulmanNotesHigherCategories2016}.

\begin{definition}
Let \(\M\) be a multicategory.
   A \emph{cartesian structure}
on $\M$ associates to a selection
    \(f \colon \fset{n} \to \fset{m} \)
    and a choice of objects 
\(A_1,\ldots, A_n, B \in \ob{\M}\)
a map 
     \[
         - \cdot f \colon
         \M(A_{f[1]}, \ldots, A_{f[m]}; B)
         \to \M(A_1, \ldots, A_n; B),
    \] 
    such that the following properties
hold:
    \begin{gather*} 
        (\varphi \cdot g) \cdot f
            =   \varphi \cdot (g \circ f),
            \qquad 
    	\varphi \cdot \mathrm{id}_{\fset n} = \varphi,\\
            (\varphi \cdot f)
            \circ 
            (\psi_1 \cdot g_1, 
            \ldots, 
            \psi_n \cdot g_n)
	= 
            \big[
            \varphi \circ 
            (\psi_{f[1]}, \ldots,
\psi_{f[m]}) 
            \big]            \cdot 
            \big[f \wr (g_1, \ldots, g_n) \big].
    \end{gather*}
\end{definition}
A \emph{cartesian multicategory} is a
multicategory with a chosen cartesian
structure.
\begin{definition}
    A \emph{cartesian operad} is a cartesian multicategory with 
    a single object.
\end{definition}

\begin{figure}[h]
    \centering
    \tikzset{every picture/.style={
    scale=0.3, 
baseline={([yshift=-.5ex]current bounding box.center)},
    every picture/.style={}}}
\begin{tikzpicture}
		\node[label={left:$A_1$}] (0) at (-1, 1) {};
		\node  (1) at (1, 1) {};
		\node[label={left:$A_1$}] (2) at (-1, -1) {};
		\node  (3) at (1, -1) {};
		\node[label={left:$A_2$}] (4) at (-1, -3) {};
		\node  (5) at (1, -3) {};
		\node  (6) at (5, -1) {};
		\node[label={right:$B$}] (7) at (7, -1) {};
		\node  (8) at (2.5, -1) {$\varphi$};
		\node  (9) at (1, 2) {};
		\node  (10) at (1, -4) {};

		\draw  (0.center) to (1.center);
		\draw  (2.center) to (3.center);
		\draw  (4.center) to (5.center);
		\draw  (6.center) to (7.center);
		\draw  (9.center) to (10.center);
		\draw  (10.center) to (6.center);
		\draw  (6.center) to (9.center);
\end{tikzpicture}
\(\quad
\longmapsto
\quad\)
\begin{tikzpicture}
\tikzset{
Black node/.style={
    fill=black,
    shape=circle,
    inner sep=0pt,
    minimum size=6pt
    }
}

		\node[style=Black node] (0) at (-1, 1) {};
		\node (1) at (1, 1) {};
		\node[style=Black node] (2) at (-1, -1) {};
		\node (3) at (1, -1) {};
		\node[style=Black node] (4) at (-1, -3) {};
		\node (5) at (1, -3) {};
		\node (6) at (5, -1) {};
		\node[label={right:$B$}] (7) at (8, -1) {};
		\node (8) at (2.5, -1) {$\varphi$};
		\node (9) at (1, 2) {};
		\node (10) at (1, -4) {};
		\node[style=Black node] (11) at (-3, 0) {};
		\node[style=Black node] (12) at (-3, -2) {};
		\node[label={left:$A_1$}] (13) at (-6, 0) {};
		\node[label={left:$A_2$}] (14) at (-6, -2) {};
		\node (15) at (-4, 3) {};
		\node (16) at (-4, -5) {};
		\node (17) at (2, -5) {};
		\node (18) at (2, 3) {};
		\node (19) at (7, -1) {};
		\node (20) at (-4, 0) {};
		\node (21) at (-4, -2) {};
		\draw (0) to (1.center);
		\draw (2) to (3.center);
		\draw (4) to (5.center);
		\draw (6.center) to (7.center);
		\draw (9.center) to (10.center);
		\draw (10.center) to (6.center);
		\draw (6.center) to (9.center);
		\draw (11) to (0);
		\draw (11) to (2);
		\draw (12) to (4);
		\draw (15.center) to (16.center);
		\draw (16.center) to (17.center);
		\draw (17.center) to (19.center);
		\draw (15.center) to (18.center);
		\draw (18.center) to (19.center);
		\draw (13.center) to (20.center);
		\draw (14.center) to (21.center);
		\draw (20.center) to (11);
		\draw (21.center) to (12);
\end{tikzpicture}
    \caption{\(- \cdot f: \M(A_1,A_1,A_2; B) \to \M(A_1,A_2;B)\)}
\end{figure}

\begin{proposition}
If \(\C\) is a cartesian category
and \(M_0 \subseteq \ob{\C}\),
then the associated multicategory is cartesian.
\end{proposition}
\begin{proof}
    For any objects \(A_1, \ldots, A_n, B\) in  \(\C\)
    and selection 
    \(f \colon \fset{n} \to \fset{m} \),
    precomposing with the morphism 
    \(\pi^{(f)}_{A_1,\ldots,A_n}\) 
    induces a map
    \[
        - \cdot f
        \colon
        \C(A_{f[1]} \otimes \cdots \otimes A_{f[m]}, B)
        \to
        \C(A_1 \otimes \cdots \otimes A_n, B)
    \] 
with the desired properties. 
\end{proof}

The converse does not hold in general;
however, we have:

\begin{proposition}
    A monoidal category \(\C\) 
    is cartesian
    if and only if
    its underlying multicategory
    is cartesian.
\end{proposition}
\begin{proof}
    One implication follows from the previous proposition.
    For the reverse implication, let us assume that the
    underlying multicategory is cartesian.
    Then, we define the uniform deletion \(\varepsilon\)
    and the uniform copying \(\Delta\) by
    setting
    \( \varepsilon_X
    \coloneqq 
    \id_\one
    \cdot
    \varepsilon_{\fset 1}
    \)
    and
    \(\Delta_X \coloneqq 
    \id_{X \otimes X} 
    \cdot 
    \Delta_{\fset 1}\)
    for each \(X \in \ob{\C}\).
\end{proof}

\begin{corollary}
    The endomorphism operad of
an object in a 
    cartesian category
is cartesian.
\end{corollary}

\subsection{Clones}
Clones are usually defined as follows:

\begin{definition}\label{defclone}
    An \emph{(abstract) clone} is a collection
    of sets \(\{ C_{n} \}_{n \in \N}\), 
    together with elements
    \[
    \pi_{i,n} \in C_{n}, \quad i=1,\ldots,n
    \] 
    and maps
    \[
        \bullet: C_{m} \times (C_{n})^{m} \to C_{n}
    \] 
    satisfying for all \(\varphi \in C_{m}\),
    \(\psi_{1}, \ldots, \psi_{m} \in C_{n}\),
    \(\rho_{1}, \ldots, \rho_{n} \in C_{l}\)
    \begin{enumerate}
        \item \(\varphi \bullet (\pi_{1,n}, \ldots, \pi_{n,n}) = \varphi\),
        \item \(\pi_{i,m} \bullet (\psi_{1}, \ldots, \psi_{m}) = \psi_{i}\), and
        \item \(\varphi \bullet (
        \psi_{1} \bullet (\rho_{1}, \ldots, \rho_{n}), 
        \ldots,
        \psi_{m} \bullet (\rho_{1}, \ldots, \rho_{n})) \\
        = ( \varphi \bullet (\psi_{1}, \ldots, \psi_{m}))
        \bullet (\rho_{1}, \ldots, \rho_{n})\)
    \end{enumerate}
\end{definition}
As is pointed out e.g.~in
\cite{gouldCoherenceCategorifiedOperadic2008,hylandNotionLambdaMonoid2014}, 
this concept is just an equivalent way to 
view cartesian operads:

\begin{theorem}\label{important2}
Let $\OO$ be a cartesian operad with
object $X$.
For any
            \(m,n \in \N,1 \le i \le n\),
            \(\varphi \in \OO_{m}\),
            \(\psi_{1}, \ldots, \psi_{m}
\in \OO_{n}\), define
            \[
            \pi_{i,n} \coloneqq 
	    \mathrm{id}_X 
            \cdot 
            \pi^i_{\fset 1, \ldots, \fset 1} 
            ,\quad
            \varphi \bullet 
            (\psi_{1}, \ldots, \psi_{m}) 
            \coloneqq 
	    \big(\varphi \circ 
                (\psi_{1}, \ldots,
                \psi_{m})
            \big)            \cdot
            \Delta^{m-1}_{\fset n},
            \] 
            where the map 
            \(\Delta^{m-1}_{\fset n} 
            \colon
            \fset n
            \to
            \fset{m n} = 
	    \fset n \oplus \cdots
	    \oplus \fset n \)
            is the uniform copying in 
$\F^\op$ applied \(m-1\) times.  
    With these operations, 
$\{\OO_n\}_{n \in \mathbb{N} }$ 
becomes an abstract clone, and this
defines an isomorphism between the
categories of cartesian operads
    and of abstract clones.
\end{theorem}
\begin{proof}
The inverse of the isomorphism is as
follows: let \(C\) be an abstract clone.
For each \(n \in \N\), let \(\OO_n \coloneqq C_n\),
for any selection \(f \colon \fset{n} \to \fset{m}\),
             let \(- \cdot f 
             \colon
             \OO_{m} \to \OO_{n}\) with
             \[
              \varphi \cdot f 
              \coloneqq 
             \varphi \bullet
(\pi_{f[1],n}, \ldots, \pi_{f[m],n}),
             \] 
and for \(\varphi \in \OO_{n}\)
             and 
             \(\psi_{i} \in \OO_{m_i}\)
             \((i=1, \ldots, n)\), let
              \[
                  \varphi \circ (\psi_{1}, \ldots, \psi_{n})
                  \coloneqq 
              \varphi \bullet (
              \psi_{1}
              \cdot 
              \pi^1_{\fset m_1, \ldots ,\fset m_n},
              \ldots,
              \psi_{n}
              \cdot 
              \pi^n_{\fset m_1, \ldots ,\fset m_n} 
              ),
             \] 
             where the map 
             \(\pi^i_{\fset m_1, \ldots \fset m_n}
             \colon
	    \fset{m_1} \oplus \ldots \oplus 
	    \fset{m_n}
             \to  \fset{m_i}
             \)
             is the canonical projection in \(\F^\op\).
             We refer the reader to
\cite{gouldCoherenceCategorifiedOperadic2008,
hylandNotionLambdaMonoid2014}
 for the details of the proof.
\end{proof}

\begin{figure}
    \captionsetup[subfigure]{labelformat=empty}
    \tikzset{
        every picture/.style={
            scale=0.4, 
every picture/.style={}}}
    \begin{subfigure}[b]{0.3\textwidth}
        \centering
        \begin{tikzpicture}
\tikzset{
Black node/.style={
    fill=black,
    shape=circle,
    inner sep=0pt,
    minimum size=6pt
    }
}

		\node[style=Black node] (0) at (-2, 1) {};
		\node[style=Black node] (1) at (-2, 0) {};
		\node[style=Black node] (2) at (-2, -1) {};
		\node[style=Black node] (3) at (0, 0) {};
		\node (4) at (-4, 1) {};
		\node (5) at (-4, 0) {};
		\node (6) at (-4, -1) {};
		\node (7) at (1.5, 0) {};
		\node (8) at (-3, 2) {};
		\node (9) at (-3, -2) {};
		\node (10) at (0, -2) {};
		\node (11) at (0, 2) {};
		\node (12) at (3, 0) {};
		\draw (4.center) to (0);
		\draw (5.center) to (1);
		\draw (6.center) to (2);
		\draw (2) to (3);
		\draw (3) to (12.center);
		\draw (8.center) to (9.center);
		\draw (9.center) to (10.center);
		\draw (10.center) to (7.center);
		\draw (7.center) to (11.center);
		\draw (11.center) to (8.center);

\end{tikzpicture}
    
        \caption{\(
            \pi_{3,3} 
\) }
    \end{subfigure}
    \hfill
    \begin{subfigure}[b]{0.5\textwidth}
        \centering
        \begin{tikzpicture}
\tikzset{
Black node/.style={
    fill=black,
    shape=circle,
    inner sep=0pt,
    minimum size=6pt
    }
}

		\node[style=Black node] (0) at (-1, 1) {};
		\node[style=Black node] (1) at (-1, 0) {};
		\node[style=Black node] (2) at (-1, -1) {};
		\node[style=Black node] (3) at (2, 3) {};
		\node[style=Black node] (4) at (2, 2) {};
		\node[style=Black node] (5) at (2, 1) {};
		\node[style=Black node] (6) at (2, -1) {};
		\node[style=Black node] (7) at (2, -2) {};
		\node[style=Black node] (8) at (2, -3) {};
		\node (9) at (3, 3) {};
		\node (10) at (3, 2) {};
		\node (11) at (3, 1) {};
		\node (12) at (3, 3.75) {};
		\node (13) at (3, 0.25) {};
		\node (14) at (5, 2) {};
		\node (15) at (3, -1) {};
		\node (16) at (3, -2) {};
		\node (17) at (3, -3) {};
		\node (18) at (3, -0.25) {};
		\node (19) at (3, -3.75) {};
		\node (20) at (5, -2) {};
		\node (21) at (7, 1) {};
		\node (22) at (7, -1) {};
		\node (23) at (7, 1.75) {};
		\node (24) at (7, -1.75) {};
		\node (25) at (9, 0) {};
		\node (26) at (-3, 1) {};
		\node (27) at (-3, 0) {};
		\node (28) at (-3, -1) {};
		\node (29) at (12, 0) {};
		\node (30) at (-2, 4.5) {};
		\node (31) at (-2, -4.5) {};
		\node (32) at (7, -4.5) {};
		\node (33) at (7, 4.5) {};
		\node (34) at (11, 0) {};
		\node (35) at (3.75, 2) {$\psi_1$};
		\node (36) at (3.75, -2) {$\psi_2$};
		\node (37) at (7.75, 0) {$\varphi$};
		\draw (0) to (3);
		\draw (1) to (4);
		\draw (2) to (5);
		\draw (0) to (6);
		\draw (1) to (7);
		\draw (2) to (8);
		\draw (3) to (9.center);
		\draw (4) to (10.center);
		\draw (5) to (11.center);
		\draw (12.center) to (13.center);
		\draw (13.center) to (14.center);
		\draw (12.center) to (14.center);
		\draw (18.center) to (19.center);
		\draw (19.center) to (20.center);
		\draw (18.center) to (20.center);
		\draw (6) to (15.center);
		\draw (7) to (16.center);
		\draw (8) to (17.center);
		\draw (23.center) to (24.center);
		\draw (24.center) to (25.center);
		\draw (25.center) to (23.center);
		\draw[in=-180, out=0, looseness=1.25] (14.center) to (21.center);
		\draw[in=-180, out=0, looseness=1.25] (20.center) to (22.center);
		\draw (26.center) to (0);
		\draw (27.center) to (1);
		\draw (28.center) to (2);
		\draw (30.center) to (31.center);
		\draw (31.center) to (32.center);
		\draw (32.center) to (34.center);
		\draw (30.center) to (33.center);
		\draw (33.center) to (34.center);
		\draw (25.center) to (29.center);

\end{tikzpicture}
    
        \caption{\(
            \varphi 
            \bullet 
            (\psi_1,\psi_2)
        \) }
    \end{subfigure}
    \caption{From a cartesian operad into a clone}
\end{figure}

\begin{remark}
\emph{Lawvere theories} are yet another
perspective on the same concepts. The
equivalence is most directly seen when
defining them as identity-on-objects functors
on $\F^\op$ that preserve products. We
refer e.g.~to
\cite{gouldCoherenceCategorifiedOperadic2008,
hylandNotionLambdaMonoid2014,lawvereFunctorialSemanticsAlgebraic2004} for
further information.  
\end{remark}

We will now turn to 
our main aim to construct an 
object $X$ in a monoidal
category $\C$ for which 
$\{X^{\otimes n}\} \subseteq
\C$ is not cartesian, but whose
endomorphism operad is
cartesian.  
Our last theoretical result is 
that for Hopf monoids
\cite[Definition 6.31]{heunenCategoriesQuantumTheory2019},
this 
phenomenon cannot occur: 

\begin{proposition}\label{hopf}
Let $H$ be a Hopf monoid in a
braided monoidal category
$\C$. Then the maps {\rm
(\ref{clonestr1})}
and {\rm (\ref{clonestr2})}
given in the introduction turn 
the endomorphism
operad of $H$ in $\ComC$ into a clone  
if and only if $H$ is
cocommutative.
\end{proposition}
\begin{proof}
The implication 
``$ \Leftarrow$'' follows from
Theorem~\ref{notsogood},
so we prove ``$\Rightarrow$''. 
A Hopf monoid is in
particular a bimonoid, that
is, a monoid in the monoidal
category $\ComC$. Thus $H$ comes
equipped with a multiplication
$ \mu \colon H \otimes H
\rightarrow H$ which is a
morphism of comonoids.
Using the clone product, this
means that  
$ \mu \bullet ( \mathrm{id}
_H, \mathrm{id} _H) = \mu \circ \Delta 
\colon 
H \rightarrow H$ is a morphism
of comonoids, too. 

Explicitly,
this means that 
$$
 	\quad ( \mu \otimes \mu ) \circ 
	( \Delta \otimes \Delta ) \circ 
	\Delta 
=  \Delta \circ \mu \circ
\Delta .
$$
Inserting the fact that 
$ \mu $ is a comonoid morphism
yields 
\begin{equation}
	\label{prague}
  ( \mu \otimes \mu ) \circ 
	( \Delta \otimes \Delta ) \circ 
	\Delta 
=  
	(\mu \otimes \mu ) \circ 
	( \mathrm{id} _H \otimes 
	\sigma _{H,H} \otimes 
	\mathrm{id} _H) \circ
	( \Delta \otimes \Delta )\circ
\Delta .
\end{equation}
To deduce from this the
cocommutativity of $ \Delta $,
we recall that 
the bimonoid structure on 
$H$ defines a second monoid  
structure on the set $\C(H,H)$
(besides the one given by
composition), whose product is 
the convolution
$$
	\varphi \conv 
	\psi  := 
	\mu \circ ( \varphi \otimes 
	\psi ) \circ \Delta
$$
and whose unit element is 
$\eta \circ
\varepsilon $ (the composition
of counit and unit morphism).
The sets 
$\C(H^{\otimes m},H^{\otimes
n})$ carry commuting left and
right actions of the 
monoid $(\C(H,H),\conv,
\eta \circ \varepsilon )$
given by  
\begin{align*}
	\varphi \lact \alpha 
	&:= 
	(\mu \otimes 
	\mathrm{id}_{H^{\otimes n-1}} ) 
	\circ 
	(\varphi \otimes \alpha) 
	\circ (\Delta \otimes 
	\mathrm{id} _{H^{\otimes
m-1}}), \\
	\alpha \ract \psi 
	&:= 
	( 
	\mathrm{id}_{H^{\otimes n-1}} 
	\otimes \mu ) 
	\circ 
	(\alpha \otimes \psi ) 
	\circ (
	\mathrm{id} _{H^{\otimes
m-1}} \otimes \Delta).
\end{align*}
With this notation, 
(\ref{prague}) can be
rewritten as 
$$
	\mathrm{id} _H \lact 
	\Delta \ract \mathrm{id}
_H= \mathrm{id} _H \lact 
	(\sigma _{H,H} \circ \Delta
) \ract \mathrm{id} _H. 
$$
Finally, a Hopf
monoid is by definition a bimonoid for
which $ \mathrm{id} _H \in 
\C(H,H)$ is
invertible with respect to
$\conv$, so the above implies 
\[
	\Delta = \sigma _{H,H} 
	\circ \Delta . \qedhere
\]
\end{proof}

\section{An example}\label{examplessec}
Let 
\(\Ab\) be the symmetric monoidal 
category of abelian groups with $ \otimes
_\mathbb{Z}$ as monoidal structure.
The category of comonoids in the opposite category 
$\Ab^\op$ is $\NCS$, the
opposite
of the category of
unital associative rings. 
Our main aim is to provide an
explicit example of a ring which viewed as a
comonoid in $\Ab^\op$ does not generate a
cartesian category, yet its 
endomorphism
operad is a clone. 

\subsection{Motivation}\label{motivsec}
In
this Section~\ref{motivsec} we
indicate why and how we were
searching for such a
ring. Readers who are 
just interested in the example
itself can safely skip this
material. 

Recall that in algebraic geometry, the 
full subcategory 
$$ 
	\Aff := \cCom(\Ab^\op)
	\subset \NCS
$$ 
of commutative
rings gets interpreted as the
cartesian category of affine
schemes, by
identifying a
commutative ring $A$ 
with the locally ringed space
$X=\mathrm{Spec}(A)$.
A morphism of commutative
rings ${A \rightarrow A^{\otimes n}}$ 
is then the same as 
a morphism of affine
schemes ${X \times \cdots \times X
\rightarrow X}$. 

What is special about 
$\Ab^\op$ is that it is
monoidal abelian. This allows 
us to consider deformations of 
cocommutative comonoids which
are not necessarily
cocommutative,
but are in a sense not
far from being so:

\begin{definition}\label{defdeformation}
An \emph{infinitesimal
deformation of order $n$} of a ring $B$ is
a surjective ring morphism 
$ A \rightarrow B$ whose
kernel $I \lhd A$ is an 
ideal for which $I^n=0$, that
is, for which $a_1 \dots
a_n=0$ for all $ a_1,\dots,a_n
\in I$.   
\end{definition}

An infinitesimal
deformation of order 2 is also
known as an \emph{abelian} (or
\emph{square zero}) 
\emph{extension}, while the
process of completion (taking
a limit $n \to \infty$) yields 
\emph{formal deformations}. 
For further
background reading, we refer 
the interested
reader to the literature:   
\cite[Section~12.2]{lodayAlgebraicOperads2012}
discusses the abstract theory of
deformations of algebras over
algebraic operads, 
\cite[Sections 1.5.3 and
1.5.4]{lodayCyclicHomology1998} or
\cite[Section~9.3]{weibelIntroductionHomologicalAlgebra1994}
contain
more information on the
specific example of
associative 
rings and algebras, 
\cite[Appendix~E]{lodayCyclicHomology1998} 
discusses the example of
commutative rings and
algebras, 
and 
\cite[Section
II.9]{hartshorneAlgebraicGeometry1977} provides 
classical background in
algebraic geometry.

We felt it would be natural
to test whether the
endomorphism operad of such
comonoids that are close to
being cocommutative 
have a bigger chance of being 
cartesian. The first examples of
rings that we considered
failed, however, see  
Remark~\ref{nonexample} below
for an explicit one. 
Proposition~\ref{hopf} gave a
conceptual explanation for this.   

As the computation of the
endomorphism operad of a given
ring is rather involved, our 
focus was then on finding rings
for which every ring
morphism $A \rightarrow A^{
\otimes n}$ is of the form 
$ a \mapsto 1 \otimes \cdots 
\otimes 1 \otimes \sigma (a) 
\otimes 1 \otimes \cdots
\otimes 1$ for a ring
endomorphism $ \sigma $, that
is, a ring whose endomorphism operad
is a clone and is as such generated by 
the endomorphisms of $A$.

We also failed to 
find noncommutative rings with this
property, but the attempt to
construct such rings in terms of
generators and relations has
led us to the example for the
original question described
here. 

\subsection{A deformation of $ \mathbb{Z}
[\sqrt{q}]$}
Let $q \in \mathbb{N} $ be a
natural number that we assume
is not a square, and let 
$A$ be the universal
ring with generators $t,x$ 
satisfying
$$
	t^2=q,\quad 
	tx=-xt,\quad 
	x^2=0.
$$
So $A$ is noncommutative
(since $tx = -xt$), and the
ring morphism 
$$
	A \rightarrow \mathbb{Z}
[\sqrt{q}] \cong A/I,\quad 
	I:=Ax=xA
$$
turns $A$ into a
deformation of order 2 of 
$ \mathbb{Z} [\sqrt{q}] $ as
in
Definition~\ref{defdeformation}. 

Let $\OO_n$ be the set of all 
ring morphisms $A \rightarrow A^{\otimes
n}$ and $ \pi _{i,n} \in
\OO_n$ be given by 
\[
	A \rightarrow 
	A^{\otimes n},\quad
	a \mapsto 
	1 \otimes 
	\cdots \otimes  1
	\otimes  a \otimes
	1 \otimes  \cdots
	\otimes  1. 
\]  
Then the ring $A^{\otimes n}$ is generated
by the elements
$$
	t_i:= \pi _{i,n}(t),\quad
	x_i:= \pi _{i,n}(x),
$$
and these generators satisfy 
$$
	t_it_j=t_jt_i,\quad
	t_i^2=q,\quad
	x_ix_j=x_jx_i,\quad
	x_i^2=0,
$$
$$
	t_ix_j=x_jt_i,\quad i \neq
	j,\quad
	t_ix_i = -x_it_i. 
$$
Finally, let
\(
	\varphi \bullet 
	(\psi _1,\ldots,\psi_m)\) 
be given for \( \varphi \in \OO_m,
\psi _j \in \OO_n\) by 
\begin{equation}
	\label{eqn:defnbullet}
	\mu^{m-1} _{A^{\otimes n}} \circ 
	(\psi _1 \otimes \cdots \otimes 
	\psi _m) \circ \varphi.
\end{equation}

The main result of this section is:

\begin{theorem}
	\label{thm:EndAisaclone}
The elements of $\OO_n$  
are given by 
$$
	t \mapsto 
	\pm t_d + f x_d,
	\quad
	x \mapsto 
	g x_d,
	\quad
	tx \mapsto 
	\pm g t_dx_d,
$$
where $d \in \{1,\ldots,n\}$
and $f,g \in A^{\otimes n}$.
Furthermore, 
\((\OO,\bullet, \pi _{i,n})\) is a clone
in $\Abop$. 
\end{theorem}

\begin{remark}\label{nonexample}
Taking $q$ to be $1$ yields the ring 
generated by $t,x$ satisfying  
$$
	t^2=1,\quad tx=-xt,\quad x^2=0.
$$
This is also a square zero extension of a
commutative ring, but its endomorphism
operad is not a clone. Indeed, this is a
Hopf algebra, namely the integral version of
Sweedler's 4-dimensional Hopf algebra, so
the claim follows from
Proposition~\ref{hopf}. 
Explicitly, the comultiplication is given
by  
$$
	\Delta \colon
	A \rightarrow A \otimes A, \,
	t \mapsto t \otimes t, \,
	x \mapsto 
	1 \otimes x+x \otimes t. 
$$
\end{remark}

\subsection{Proof of Theorem \ref{thm:EndAisaclone}}
As an abelian group,
$A^{\otimes n}$ is free with
basis   
$$
	t_Sx_T:=\prod_{i \in S} t_i \prod_{j
\in T} x_j,
$$
where $S,T$ run through all
pairs of subsets of 
$\{1,\ldots,n\}$. 
The product $t_Sx_Tt_Ux_V$
vanishes unless $T \cap
V=\emptyset$. In this case, 
$$
	t_Sx_T t_Ux_V =
	(-1)^{|T \cap U|} 
	q^{
	|S \cap U|}  
	t_{S 
	\triangle U}
	x_{T \cup V}.
$$ 

In order to describe the elements of 
$\OO_n$, we first show:

\begin{lemma}
The element 
$
	a = 
\sum_{ST}
	a_{ST} 
	t_Sx_T \in A^{\otimes n}
$
satisfies
$a^2=q$ if and only if 
$a=\pm t_d+fx_d$ for some 
$f \in A^{\otimes n}$. 
\end{lemma}
\begin{proof}
That the given elements square
to $q$ is clear. Assume
conversely that 
$$
a^2=\sum_{\genfrac{}{}{0pt}{}{STUV}{T \cap V=\emptyset} 
	}
	a_{ST} a_{UV}
	(-1)^{|T \cap U|} 
	q^{|S \cap U|}  
	t_{S 
	\triangle U}
	x_{T \cup V}\\
=q.	
$$
Consider first the terms in
$ \mathbb{Z} $, that is, those with 
$T=V=\emptyset,S=U$:
$$
	\sum_{S}
	a_{S \emptyset}^2 
	q^{|S|}  
	=q.	
$$
As $q$ is assumed to be 
non-square, we have 
$a_{\emptyset\emptyset}^2 \neq
q$. However, the remaining
nonzero summands are each one
greater or equal to $q$. 
Hence there exists 
some $d \in \{1,\ldots,n\}$
such that $a_{\{d\}
\emptyset} = \pm 1$ while 
$a_{S \emptyset}=0$ for all 
other $S$. 

Thus $a$ can be uniquely written as 
$\pm t_d + fx_d + g$, where 
$f,g$ are linear combinations of 
$t_Sx_T$ with
$T \neq \emptyset$ but $d \notin
T$. 
It follows that 
$t_dfx_d+fx_dt_d=0$ and $(fx_d)^2=0$,  
so
$$
	a^2 =	
	q \pm 2t_dg + g^2 + (fx_dg+gfx_d).
$$ 
Therefore, $a^2=q$ implies 
$\pm 2 t_dg + g^2 + fx_dg+gfx_d = 0$.
The term $fx_dg+gfx_d$ is in the
ideal generated by $x_d$
while $ \pm 2 t_dg + g^2$ is a
linear combination of 
$t_Sx_T$ with $d \notin T$. 
Thus the two terms vanish
separately. 

Note that $t_d$ and $g$ commute and that
$g$ is nilpotent, so there exists $m \ge
0$ such that $g^m=0$. Thus we have 
\begin{align*}
& \ \pm 2t_d g + g^2 = (\pm 2 t_d + g) g = 0 \\
\Rightarrow 
& \ (4 q - g^2) g = (\pm 2 t_d - g)(\pm 2t_d  + g) g =0 \\
\Rightarrow 
& \ (16 q^2 - g^4 ) g = (4q + g^2) (4q -g^2) g = 0 \\
\Rightarrow 
& \ldots \\
\Rightarrow 
& \ 2^{2m} q^m g = 0
\end{align*}
so $g=0$. 
\end{proof}

\begin{lemma}
If $f \in A^{\otimes n}$ is arbitrary,
then 
$T:=\pm t_d + fx_d$ and 
$X \in A^{\otimes n}$ 
anticommute if and only if  
$X \in A^{\otimes
n}x_d$.
\end{lemma}
\begin{proof}
Without loss of generality we may assume
$X$ is in the span of $t_Ux_V$ with $d
\notin V$. 
Then we have 
\begin{align*}
	0 &= 
	TX+XT \\
&= 
	\pm t_d X \pm X t_d +
	fx_dX+Xfx_d
	\\
&= 
	\pm 2 t_d X +
	fx_dX+Xfx_d.
\end{align*}
The last two terms are in the ideal
generated by $x_d$, so by our assumption
on $X$, the first term vanishes
separately, hence $t_dX=0$. Multiplication by
$t_d$ yields the claim. 
\end{proof}

As the elements of $A^{\otimes n} x_d$ all
square to zero, the above two lemmata
imply the first half of
Theorem~\ref{thm:EndAisaclone}. 

To prove the second half, 
it suffices  to show that the 
$ \mathbb{Z} $-linear map defined in
(\ref{eqn:defnbullet})
is a ring morphism.

Let us start with the following observation:
\begin{lemma}
			Let \(d \in \{1,\ldots,n\}\) and 
	$ \tau_{d} \colon A^{\otimes n} \rightarrow 
A^{\otimes n}$ be the unique ring 
morphism such that $ \tau_{d} (t_d)=-t_d, \tau_{d}
(x_d)=0$, $ \tau_{d} (t_i) = t_i $ and 
$ \tau_{d} (x_i) = x_i $ for $i \neq d$. Then
for all \(a \in A^{\otimes n}\), we have 
			\(x_{d} a = \tau_{d}(a) x_{d}\) and
			\(t_{d} \tau_{d} ( a) = \tau_{d} (a)
t_{d}\).
\end{lemma}
\begin{proof}
By linearity of both expressions,
it suffices to verify the claim for 
$a=  t_Sx_T$ which is straightforward; we
have $ \tau_{d} (t_Sx_T) = 0$ if $d \in T$, 
have $ \tau_{d} (t_Sx_T) = -t_Sx_T$ if $d \in
S$, and $\tau_{d}(t_Sx_T) =t_Sx_T$ otherwise.
\end{proof}
\begin{lemma}
Let \(d \in \{1,\ldots, m\}\), 
			\(\psi_{1}, \ldots,
\psi_{m} \colon A \to A^{\otimes n}\) be
ring morphisms and assume that 
\(e \in \{1, \ldots, n\}\) and \(f \in
A^{\otimes n}\) are such that 
\(\psi_{d}(t) \coloneqq t_{e} + f x_{e}\).
Then,
\begin{enumerate}
	\item \(\mu^{m-1} \circ 
			(\psi_{1} \otimes \cdots \otimes \psi_{m})
			(t_{d}) = t_{e} + fx_{e} \).
		\item For any \(h \in A^{\otimes m}\),
		there exists an element
		\(h' \in A^{\otimes n}\) with
		\begin{align*}
			\mu^{m-1} \circ 
			(\psi_{1} \otimes \cdots \otimes \psi_{m})
			(h x_{d})
			& = h' x_{e} \\
			\mu^{m-1} \circ 
			(\psi_{1} \otimes \cdots \otimes \psi_{m})
			(h t_{d}x_{d})
			& = h' t_{e}x_{e}.
		\end{align*}
\end{enumerate}
\end{lemma}
\begin{proof}
\begin{enumerate}
	\item We have
\begin{align*}
	(\psi_{1} \otimes \cdots \otimes \psi_{m})
	(t_{d}) & = 
	(\psi_{1} \otimes \cdots
	\otimes \psi_{m})
	(1 \otimes \cdots
	\otimes t
	\otimes \cdots
	\otimes 1) \\
	& =
	\psi_{1}(1) \otimes \cdots
	\otimes \psi_{d}(t)
	\otimes \cdots
	\otimes \psi_{m}(1) \\
	& = 1 \otimes \cdots
	\otimes \psi_{d}(t)
	\otimes \cdots
	\otimes 1,
\end{align*}
hence
\[
	\mu^{m-1} \circ
	(\psi_{1} \otimes \cdots \otimes \psi_{m})
	(t_{d}) = 
	\psi_{d}(t) = t_{e} + f x_{e}.
\] 

	\item Let \(g \in A^{\otimes n}\) such that
\(\psi_{d}(x) \coloneqq g x_{e}\)
and 
\(\psi_{d}(tx) \coloneqq g t_{e} x_{e}\).
		Without loss of generality, we may assume that
		\(h = h_1 \otimes \cdots \otimes h_{m}\)
		for some \(h_{1}, \ldots, h_{m} \in A\).
		We have
		\begin{align*}
			(\psi_{1} \otimes \cdots \otimes \psi_{m})
			(h x_{d})
			& =
			(\psi_{1} \otimes \cdots \otimes \psi_{m})
			(h_1 \otimes \cdots 
			\otimes h_{d} x
			\otimes \cdots
			\otimes h_{m}) \\
			&=
			\psi_{1}(h_{1}) 
			\otimes \cdots \otimes 
			\psi_{d}(h_{d} x)
			\otimes \cdots \otimes 
			\psi_{m}(h_{m}) \\
			& = 
			\psi_{1}(h_{1}) 
			\otimes \cdots \otimes 
			\psi_{d}(h_{d}) 
			g x_{e}
			\otimes \cdots \otimes 
			\psi_{m}(h_{m}) 
		.\end{align*}
		Then
		\[
			\mu^{m-1} \circ
			(\psi_{1} \otimes \cdots \otimes \psi_{m})
			(h x_{d})
			 =
			a x_{e} b
			=
			a \tau_{e}(b) x_{e},
		\]
		where
		\[
		a = 
			\psi_{1}(h_{1}) 
			\cdots 
			\psi_{d}(h_{d}) 
			g,
			\quad \text{and} \quad
		b = 
			\psi_{d+1}(h_{d+1})
			\cdots 
			\psi_{m}(h_{m}).
		\] 
Here \(\tau_{e}\) is as in the previous lemma.

		Following the same process, we obtain that
\[
			(\psi_{1} \otimes \cdots \otimes \psi_{m})
			(h t_{d}x_{d})
			=
			\psi_{1}(h_{1}) 
			\otimes \cdots \otimes 
			\psi_{d}(h_{d}) 
			g t_{e}x_{e}
			\otimes \cdots \otimes 
			\psi_{m}(h_{m}),
		\]
		hence
		\[
			\mu^{m-1} \circ
			(\psi_{1} \otimes \cdots \otimes \psi_{m})
			(h t_{d}x_{d})
			= a t_{e} x_{e} b
			= a \tau_{e}(b) t_{e} x_{e}.
		\]
\end{enumerate}
\end{proof}

With this we are ready to prove the 
second half of Theorem \ref{thm:EndAisaclone}:
\begin{proof}[End of proof of Theorem \ref{thm:EndAisaclone}]
Let \(\varphi \colon A \to A^{\otimes m}\)
and
\(\psi_{1}, \ldots, \psi_{m} \colon A \to A^{\otimes n}\)
be ring morphisms, and let \(\chi \coloneqq 
	\mu^{m-1} \circ 
	(\psi _1 \otimes \cdots \otimes 
	\psi _m)\).
	Our goal is to show that the \(\mathbb{Z}\)-linear map
\(\chi \circ \varphi\)
is again a ring morphism. 
Suppose that
\[
		\varphi(t) = \pm t_{d} + f x_{d},
		\quad
		\varphi(x) = g x_{d},
		\quad
		\varphi(tx) = g t_{d} x_{d}.
\] 
for some \(d \in \{1, \ldots, m\}\) 
and \(f,g \in A^{\otimes m}\),
and that \(\psi_{d}(t) \coloneqq  t_{e} + hx_{e}\) 
for some \(e \in \{1,\ldots,n\}\) and  \(h \in A^{\otimes n}\).

Then
\begin{align*}
	\chi \circ \varphi(t)
	& = \pm \chi(t_{d}) + \chi(f x_{d})
	= \pm t_{e} + h x_{e} + f' x_{e}, \\
	\chi \circ \varphi(x) 
	& = \chi(g x_{d})
	= g' x_{e}, \\
	\chi \circ \varphi(tx)
	& = \chi(g t_{d}x_{d})
	= g' t_{e}x_{e},
\end{align*}
where \(f',g' \in A^{\otimes n}\) are as in the previous lemma,
hence \(\chi \circ \varphi\) is a ring morphism.
\end{proof}


\section{Appendix: Counital comagmas}
In this appendix we prove many
of the preliminary results
on comagmas using the graphical
language of string diagrams
when representing morphisms 
in $\C$. We read these from
top to bottom, so a
comultiplication 
$ \Delta_X \colon X \rightarrow
X \otimes X$, a counit 
$ \varepsilon_X \colon X
\rightarrow \one$ and the
braiding $ \sigma _{X,Y}
\colon X \otimes Y \rightarrow
Y \otimes X$ will be depicted
as follows: 
$$
\varepsilon_{X} \coloneqq 
.
$$
See
e.g.~\cite{heunenCategoriesQuantumTheory2019}
for more information.
The
comultiplication and counit on a 
tensor product of comagmas is
depicted as follows:
\begin{equation}\label{counittensor}
%
.
\qedhere
\end{equation}
\setcounter{section}{2}
\setcounter{theorem}{12}
\begin{proposition}
    Let \(X,Y\) be two counital comagmas
in $\C$.
    The following equalities
    hold in \(\C\):
    \begin{enumerate}
        \item 
    \(
        [(\id_X \otimes \varepsilon_Y)
       \otimes
       (\varepsilon_X \otimes \id_Y)]
       \circ
       \Delta_{X \otimes Y}
        = \id_X \otimes \id_Y.
    \)
       \item
           \(
        [ (\varepsilon_X \otimes \id_Y)
       \otimes
       (\id_X \otimes \varepsilon_Y)]
       \circ
       \Delta_{X \otimes Y}
       = \sigma_{X,Y}.
           \)

    \end{enumerate}
\end{proposition}
\begin{proof}
\begin{equation}\label{twoproperties}
%
\qedhere
    \end{equation}

\end{proof}

\begin{proposition}
    \label{prop:notbraided}
    Let \(X,Y\) be two counital comagmas in  \(\C\).
    The braiding  \(\sigma_{X,Y}: X \otimes Y \to Y \otimes X\)
     is a morphism of counital comagmas
     if and only if \(\sigma_{Y,X} = \sigma_{X,Y}^{-1} \).
\end{proposition}
\begin{proof}
    If  \(\sigma_{X,Y}\) is a morphism of comagmas,
    then 
    \begin{center}
    \(\displaystyle
\Delta_{Y \otimes X} \circ \sigma_{X,Y}  = 
%
 = 
(\sigma_{X,Y} \otimes \sigma_{X,Y})
\circ \Delta_{X \otimes Y}.
\)
    \end{center}
    This implies that
    \begin{center}
        \(\displaystyle
%
\).
        \end{center}
\end{proof}

\setcounter{theorem}{16}
\begin{proposition}
  Let
  \(
    X_{1} \coloneqq~(X, \Delta_{1}, \varepsilon_{1})
    \) and
\(X_{2} \coloneqq~(X, \Delta_{2}, \varepsilon_{2})
  \)
   be two counital comagma structures
   defined
   on the same underlying object \(X\).
   If \(\Delta_{2}\) is a morphism of
   comagmas \(X_{1}
\rightarrow X_{1} \otimes X_{1}\),
   then
\[
   \varepsilon_{1} =
   \varepsilon_{2},
   \quad
   \Delta_{1} =
   \Delta_{2},
\]
      and \(X_1=X_2\)
      is a cocommutative
comonoid.
\end{proposition}

\begin{proof}
  \label{eckmannhiltonproof}
    Let \((\Delta_1, \varepsilon_1) =\) 
    \( 
\left(

\right)\),
    then \(\Delta_{2}\) being a morphism of comagmas means that
    \begin{center}
    \(\displaystyle
    %
\).
    \end{center}

    The equality for the counit is obtained as follows:
    \begin{center}
        \(\displaystyle
%
\).
    \end{center}
    From here onwards we will omit the colours for the counit.
    Next comes the equality for the comultiplications:
    \begin{center}
    \(\displaystyle
%
\).
    \end{center}
    Now we shall remove the colours for the comultiplications.
    For the coassociativity, we have:
    \begin{center}
    \(\displaystyle
%
\),
    \end{center}
    and for the cocommutativity:
    \begin{center}
    \(\displaystyle
%
\).
    \end{center}
    
\end{proof}

\setcounter{section}{5}
\setcounter{theorem}{0} 

\begin{proposition}
    \label{prop:hilton-eckmann-appendix2}
    If a comonoid 
\(X\) 
is cocommutative,
    then  its comultiplication \(\Delta_X \colon X \to X \otimes X\)
    is a morphism in $\ComC$.
\end{proposition}
\begin{proof}
    The coassociativity implies that
    \begin{center}
         \(\displaystyle
%
\).
    \end{center}
    Combined with the cocommutativity, it yields: 
    \begin{center}
    \tikzset{every picture/.style={scale=0.8, every picture/.style={}}}
         \(\displaystyle
(\Delta_{X} \otimes \Delta_{X}) 
    \circ \Delta_{X}  = 
%
 =  
\Delta_{X \otimes X} \otimes \Delta_{X}.
\)
    \end{center}
\end{proof}

\begin{proposition}
    \label{prop:tens-prop-cocommutative}
    Let \(X,Y\) be two comonoids.
    Their tensor product \(X \otimes Y\)
    is cocommutative if and only if
    the three conditions 
    in Figure~\ref{fig:cocom} hold.
    \begin{figure}[h!]
    \captionsetup[subfigure]{textformat=empty}
\centering
    \begin{subfigure}[b]{0.3\textwidth}
    \centering
    \(\displaystyle
%
\)
    \caption{~}
    \label{eq:cocom1}
    \end{subfigure}
    \hfill
    \begin{subfigure}[b]{0.3\textwidth}
    \centering
    \(\displaystyle
%
\)
    \caption{~}
    \label{eq:cocom2}
    \end{subfigure}
    \hfill
    \begin{subfigure}[b]{0.3\textwidth}
    \centering
    \(\displaystyle
%
\)
    \caption{~}
    \label{eq:cocom3}
    \end{subfigure}
    \caption{Conditions for cocommutativity}
    \label{fig:cocom}
\end{figure}
    
\end{proposition}
\begin{proof}
 Let us assume that \(X \otimes Y\) is cocommutative,
 that is,
\begin{center}
 \(\displaystyle
\sigma_{X \otimes Y, X \otimes Y} \circ \Delta_{X \otimes Y}
= 
%
= 
\Delta_{X \otimes Y}.
\)
\end{center}
 First, we obtain (\subref{eq:cocom1}) as follows:
 \begin{center}
     \(\displaystyle
%
\).
 \end{center}
 The property (\subref{eq:cocom2}) 
 is proven in a similar manner.
 Next, (\subref{eq:cocom3}) is obtained as follows:
 \begin{center}
     \(\displaystyle
%
\).
 \end{center}

 Conversely, if  the three conditions are satisfied, then
 \begin{center}
     \(\displaystyle
%
\).
 \end{center}
\end{proof}

\end{document}